\documentclass[11pt]{article}
\usepackage[utf8]{inputenc}
\usepackage[T1]{fontenc}
\DeclareUnicodeCharacter{0229}{\k{e}}
\usepackage{lmodern}
\usepackage{subfiles}
\usepackage{enumitem}
\setenumerate{topsep=6pt,ref={\normalfont(\roman*)},label={\normalfont(\roman*)}, itemsep=0pt} 

\usepackage{amsfonts}
\usepackage{amsthm}
\usepackage{amsmath}
\usepackage{amssymb}
\usepackage{amscd}
\usepackage{mathrsfs}
\usepackage{mathtools}
\usepackage{bbm}
\usepackage{esint}

\usepackage[margin=2.6cm]{geometry}
\usepackage{setspace}
\usepackage{indentfirst}
\usepackage{graphicx}
\usepackage{graphics}
\usepackage{lscape}
\usepackage{pgf,tikz}
\usepackage{tikz-cd}
\usepackage{color}
\usepackage{pict2e}
\usepackage{epic}
\usepackage{epstopdf}
\usepackage{titlesec, titlefoot}
\titleformat{\section}[block]{\Large\bfseries\filcenter}{\thesection}{1em}{}
\titleformat{\part}[block]{\LARGE\bfseries\filcenter}{Part \thepart.}{0.5em}{}
\usepackage{commath}
\usepackage{float}
\usepackage{caption}
\usepackage{etoolbox}
\usepackage[affil-it]{authblk}
\usepackage{combelow}

\usepackage[hidelinks,bookmarksdepth=3]{hyperref}
\hypersetup{bookmarksopen=true} 

\graphicspath{{./Pictures/}}
\allowdisplaybreaks

\expandafter\def\expandafter\normalsize\expandafter{%
\normalsize
\setlength\abovedisplayskip{6pt}
\setlength\belowdisplayskip{6pt}
\setlength\abovedisplayshortskip{6pt}
\setlength\belowdisplayshortskip{6pt}
}

\theoremstyle{plain}

\renewcommand*\thesection{\arabic{section}}
\numberwithin{equation}{section} 

\newtheorem{theorem}{Theorem}[section]
\newtheorem{lemma}[theorem]{Lemma}
\newtheorem*{lemma*}{Lemma}
\newtheorem{proposition}[theorem]{Proposition}
\newtheorem{corollary}[theorem]{Corollary}

\theoremstyle{definition}
\newtheorem{definition}[theorem]{Definition}
\newtheorem{remark}[theorem]{Remark}
\newtheorem{example}[theorem]{Example}

\expandafter\let\expandafter\oldproof\csname\string\proof\endcsname
\let\oldendproof\endproof
\renewenvironment{proof}[1][\proofname]{%
\oldproof[\upshape \bfseries #1]%
}{\oldendproof}

\makeatletter
\def\@makechapterhead#1{%
\vspace*{50\p@}%
{\parindent \z@ \raggedright \normalfont
\interlinepenalty\@M
\Huge\bfseries  \thechapter.\quad #1\par\nobreak
\vskip 40\p@
}}
\makeatother

\renewcommand{\Re}{\operatorname{Re}}
\renewcommand{\Im}{\operatorname{Im}}

\newcommand{\detpm}{\mathop{\det}\nolimits^{+-}}
\newcommand{\detpp}{\mathop{\det}\nolimits^{++}}

\DeclareMathOperator{\supp}{supp}

\DeclareMathOperator{\Id}{Id}

\def \a{\alpha}

\def \R {\mathbb{R}}

\def \C{\mathbb{C}}

\def \D{\textup{D}}

\def \e{\varepsilon}
\def \d{\textup{d}}
\def \n{\nabla}

\def \p{\partial}
\def \mc{\mathcal}
\def \mb{\mathbb}

\def \tp{\textup}
\def \Id{\textup{Id}}

\def \loc{\textup{loc}}

\renewenvironment{thebibliography}[1]{
  \begin{oldthebibliography}{#1}
    \setlength{\itemsep}{0.5pt}
    \setlength{\parskip}{0.5pt}
}{
  \end{oldthebibliography}
}

\begin{document}

	\title{\textbf{Quasiconvexity of the Burkholder function\\ on symmetric matrices}}
		
	\author{{\Large Andr\'e Guerra}}
		
	\affil{\small Department of Pure Mathematics and Mathematical Statistics,  University of Cambridge,\protect\\  Wilberforce Rd, Cambridge CB3 0WB, UK
	\protect\\
	{\tt{adblg2@cam.ac.uk}} \ }
	
		\affil{\small Departamento de Matemática, Instituto Superior Técnico,  Av.\ Rovisco Pais 1049-001 Lisboa, Portugal
	\protect\\
	{\tt{andre.l.guerra@tecnico.ulisboa.pt}} \ }

	\date{}
	
	\maketitle
	
	\begin{abstract}
	We use Morse Theory to prove that the Burkholder function is quasiconvex on symmetric matrices, which implies several sharp inequalities for the Beurling--Ahlfors transform when acting on real-valued functions. More generally, we prove that the separating functions introduced by Székelyhidi are quasiconvex on symmetric matrices.
	\end{abstract}

\section{Introduction}
Identifying $\R^2$ with $\C$, every matrix
$A\in\mathbb R^{2\times2}$ can be written uniquely as
\[
Az=a_+z+a_-\overline z,
\qquad a_+,a_-\in\mathbb C,
\]
so we write $A\equiv(a_+,a_-)$. We also write $|A|\equiv|a_+|+|a_-|$ for its operator norm.  In this note we are concerned with the integrand
$L\colon \R^{2\times 2}\to \R$ defined by
\begin{equation}
\label{eq:L}
L(A)\equiv
\begin{cases}
\det A & \text{if } |A|\leq1,\\
2|a_+|-1 & \text{if } |A|>1.
\end{cases}
\end{equation}

Our main result is the following quasiconvexity inequality:

\begin{theorem}\label{thm:L}
For every $A\in\mathbb R^{2\times2}$ and every $u\in C_c^\infty(\mathbb R^2)$ we have
\[
\int_{\mathbb R^2}
\bigl[L(A+\D^2u)-L(A)\bigr]\d x\geq0.
\]
\end{theorem}

Theorem \ref{thm:L} asserts that $L$ is quasiconvex when restricted to symmetric matrices, in the sense that it satisfies Jensen's inequality with respect to any homogeneous gradient  Young measure supported in $\R^{2\times 2}_\tp{sym}$, see \cite{Sverak1992} and also \cite{Maso2003}.

We emphasize that the proof of Theorem \ref{thm:L} relies crucially on the scalar potential $u$, since we use ideas from Morse Theory.       Morse Theory is by now a well-established tool to attack such problems, see \cite{GuerraTione2025,Muller2003,Sverak1992}. We refer the reader to  \cite{Szekelyhidi2005a} for a different idea in a related problem.

The integrand $L$ was discovered by Burkholder \cite{Burkholder1989}
and, independently, by \v Sver\'ak \cite{Sverak1990}. Its quasiconvexity is a long-standing conjecture \cite{Baernstein1997a,Iwaniec2002}, since it is well-known
that quasiconvexity inequalities for $L$ imply a host of sharp integral
estimates related to the Beurling--Ahlfors transform $\mc S$ \cite{Astala2009}, and in particular the Iwaniec conjecture \cite{Iwaniec1982}. We recall that the Beurling--Ahlfors transform is
the Calder\'on--Zygmund operator such that
$$\mathcal S\circ\partial_{\bar z}=\partial_z.$$ 
Theorem~\ref{thm:L} gives the following sharp weak-type estimate for $\mc S$:

\begin{corollary}
\label{cor:beurling}
For every real-valued $h\in L^1(\mathbb C)$ and every $\lambda>0$,
\[
\lambda\bigl|\{z\in\mathbb C:|\mathcal S h(z)|>\lambda\}\bigr|
\leq2\|h\|_{L^1(\mathbb C)}.
\]
In particular,
$\|\mathcal S h\|_{L^{1,\infty}(\mathbb C)}
\leq2\|h\|_{L^1(\mathbb C)}.$
The constant $2$ is sharp.
\end{corollary}

The function $L$ was, in fact, introduced by Burkholder to prove a sharp weak-type
estimate for martingales \cite{Burkholder1989}. Weak-type
estimates for $\mc S$ were studied for classes of radial maps
\cite{Beurling2009,Banuelos2013};
somewhat surprisingly, the optimal
constant in the radial class is $1/\log 2$.  Instead, the optimizing
sequence in Corollary~\ref{cor:beurling} consists of  second-order laminates, disproving in particular \cite[Conjecture 2]{Beurling2009}. See
\cite{Boros2013,Conti2005} for related constructions.

For $1<p<\infty$, set	
\[
p^*\equiv\max\left\{p,\frac{p}{p-1}\right\}.
\]
We define the Burkholder integrand \cite{Burkholder1984} as
\[
B_p(A)
\equiv
\begin{cases}
\displaystyle
\left((p^*-1)|a_+|-|a_-|\right)
\left(|a_+|+|a_-|\right)^{p-1}
&\text{if }1<p<2,\\
\displaystyle
\left((p^*-1)|a_-|-|a_+|\right)
\left(|a_+|+|a_-|\right)^{p-1}
&\text{if } 2<p<\infty.
\end{cases}
\]
The definition at $p=2$ is immaterial, since $B_2$ is the determinant up to a sign; equivalently, $\mc S$ is an isometry in $L^2$.
It was observed in \cite{Baernstein1997a} that, to a normalization, the Mellin transform of $L$ is the Burkholder
function $B_p$, and thus Theorem~\ref{thm:L}
immediately gives the following:

\begin{corollary}
\label{cor:burkholder}
For every $p\in (1,\infty)$, $A\in\mathbb R^{2\times2}$, and
$u\in C_c^\infty(\mathbb R^2)$,
\[
\int_{\mathbb R^2}
\bigl[B_p(A+\D^2u)-B_p(A)\bigr]\d x\geq0.
\]
\end{corollary}

Further quasiconvexity inequalities for the Burkholder function, as well as accompanying lower semicontinuity theorems, were proved in \cite{Astala2022,Astala2024,Astala2012,GuerraKristensen2021}; these inequalities are valid beyond symmetric matrices, but in turn require a sign on the integrand. Corollary \ref{cor:burkholder} neither implies nor is implied by any of these results.

By an elementary inequality proved by Burkholder \cite{Burkholder1984}, Corollary \ref{cor:burkholder} implies:

\begin{corollary}
\label{cor:beurling-lp}
Let $h\in L^p(\mathbb C)$. Then
\[
\begin{aligned}
\|\mathcal S h\|_{L^p(\mathbb C)}
&\leq\frac1{p-1}\|h\|_{L^p(\mathbb C)}
&&\text{if $1<p<2$ and $h$ is real-valued},\\
\|\mathcal S h\|_{L^p(\mathbb C)}
&\leq(p-1)\|h\|_{L^p(\mathbb C)}
&&\text{if $2<p<\infty$ and $\mathcal S h$ is real-valued}.
\end{aligned}
\]
These constants are sharp.
\end{corollary}

We refer the reader to \cite{Banuelos2010,Banuelos2008,Borichev2013,Volberg2004}
for $L^p$-estimates for $\mc S$ with non-sharp constants but without
restrictions on $h$.

\medskip

Since below we will prove a more general version of Theorem \ref{thm:L}, valid for the general separating functions of \cite{Faraco2008,Szekelyhidi2005}, we believe it is helpful to sketch its proof here in the simplest case $A=0$.\footnote{A simple argument using rank-one convexity shows that quasiconvexity of $L$ at 0 implies quasiconvexity of $L$ everywhere.} Our starting point is  the key identity
\begin{equation}
\label{eq:keyid}
\begin{split}
L(X)+\det X={}&
\detpp(\Id + X)+\detpp(\Id - X) -2\\
&\qquad  - \min\{\detpm(\Id + X), \detpm(\Id-X)\},
\end{split}
\end{equation}
where $X\in \R^{2\times 2}_\tp{sym}$ and 
 $\detpp,\detpm\colon \R^{2\times 2}_\tp{sym}\to \R$ are the quasiconvex integrands introduced in \cite{Sverak1992}:
\[
\detpp(A)\equiv
\begin{cases}
\det A&\text{if } A>0,\\
0&\text{otherwise},
\end{cases}
\qquad 
\detpm(A)\equiv
\begin{cases}
\lvert\det A\rvert&\text{if } \det A<0,\\
0&\text{otherwise},
\end{cases}
\]
cf.\ Lemma \ref{lemma:decomposition} and Remark \ref{rem:so2-decomposition}. 
For a simple proof of the quasiconvexity of $\detpp$ see
\cite{Muller2003}, see \cite{Faraco2003} for an alternative approach, and \cite{Guerra2018} for further properties of these integrands.
Hence, setting
$$\phi_\pm(x)=\frac12|x|^2\pm u(x),\qquad \phi^y_\pm(x)\equiv \phi_\pm(x)-y\cdot x,$$
to prove Theorem \ref{thm:L}, we have to show that
\begin{align}
\label{eq:goal}
\begin{split}
\int_{\R^2}\left[\detpp(\D^2\phi_+)-1+\detpp(\D^2\phi_-)-1\right]
\geq
 \int_{\R^2}\min\left\{ \detpm (\D^2 \phi_+), \detpm (\D^2 \phi_-)\right\}
\end{split}
\end{align}
since the determinant is a null Lagrangian. 
For a generic \(y\), let $n_{0,\pm}(y)$  denote the numbers of  local minima  of  \(\phi_\pm^y\). Then the left-hand side in \eqref{eq:goal} has a clear geometric interpretation: by the area formula, cf.\ Lemma \ref{lemma:degree-count},
$$\int_{\mathbb R^2}
\left[\detpp(\D^2\phi_\pm)-1\right] \d x
=
\int_{\mathbb R^2}
\left[n_{0,\pm}(y)-1\right] \d y.
$$
It turns out that there is a way of rewriting the right-hand side in terms of $\detpm$. Indeed, 
a simple argument from Morse theory says that, for a generic $y$,  
$$n_{0,\pm}(y)-1=s_\pm(y),$$ where $s_\pm(y)$ denotes the number of \textit{separating saddle points} of $\phi_\pm^y$, see \cite{Herau2011} and Definition \ref{definition:separating-saddle}.  
In turn, a further application of the area formula, cf.\ Lemma \ref{lemma:area}, shows that
$$
\int_{\mathbb R^2}
s_{\pm}(y)\, \d y
=
\int_{M_\pm} \detpm(\D^2 \phi_\pm) \,\d x
$$
where $M_\pm$ is the set of separating saddle points $x$ of $\phi^{\D \phi(x)}_\pm$. Hence, combining the above three identities, \eqref{eq:goal} is equivalent to
$$\int_{M_+}\detpm(\D^2\phi_+) \,\d x + \int_{M_-}\detpm(\D^2\phi_-)\, \d x
\geq
 \int_{\R^2}\min\left\{ \detpm (\D^2 \phi_+), \detpm (\D^2 \phi_-)\right\}\d x.$$
In turn, this inequality is relatively easy to prove, as it holds pointwise a.e.:
\begin{equation*}
\label{eq:pointwiseid}
1_{M_+}\detpm(\D^2 \phi_+) + 1_{M_-} \detpm(\D^2 \phi_-)\geq \min\{\detpm(\D^2 \phi_+),\detpm(\D^2 \phi_-)\}.
\end{equation*}
The proof of this last inequality is a combination of a planar topology argument based on the Jordan Curve Theorem (Lemma \ref{lemma:planar-separation}) with the definition of separating saddle point.

\medskip

We conclude by remarking that the integrands $L_\Gamma$ defined in Section \ref{sec:sep} below, and which generalize \eqref{eq:L}, are all rank-one convex in $\R^{2\times 2}$ and play a crucial role in the study of quasiconvexity \cite{Faraco2008,Kirchheim2008,Szekelyhidi2005}.
The understanding of the possible difference between rank-one convexity and quasiconvexity in $\R^{2\times 2}$ remains a fascinating open problem going back to Morrey \cite{Morrey1952}. We refer the reader to \cite{Agazzi2026,Cassese2026,Grabovsky2018,Sverak1992a} for counter-examples in higher dimensions, \cite{Agazzi2026} for counter-examples in $\R^{3\times 3}_\tp{sym}$, and \cite{Astala2022,Astala2012,Bruno2026,Conti2003,Faraco2008,GuerraCosta2020,Harris2018,Kirchheim2008,Muller1999b,Sebestyen2017,Szekelyhidi2005} for some positive results in $\R^{2\times 2}$.

\subsection*{Notation}

If $A\equiv(a_+,a_-)$, then
\[
A=
\begin{pmatrix}
\Re a_++\Re a_-&-\Im a_++\Im a_-\\
\Im a_++\Im a_-&\Re a_+-\Re a_-
\end{pmatrix}.
\]
Thus $A\in \R^{2\times 2}_\tp{sym}$ if and only if $\Im a_+=0$. In this case, we write $A\geq0$ if $A$ is positive semidefinite and
$A>0$ if $A$ is positive definite. The inequalities $A<0,A\leq 0$ have identical meanings.  We have $|A|=|a_+|+|a_-|$ and $\det A= |a_+|^2 -|a_-|^2.$

\subsection*{Acknowledgements}
 AG acknowledges the support of the Royal Society through a Newton International Fellowship. He would also like to thank D.\ Faraco, B.\ Kirchheim, J.\ Kristensen, R.\ Tione, V. \v Sver\'ak  and L.\ Székelyhidi for many insightful conversations about this problem in the last 8 years. In particular, the insight that the integrand $L$ should be related to weak-type estimates, which led the author to the simple proof of Corollary \ref{cor:beurling}, is due to V. \v Sver\'ak. 

\subsection*{AI disclosure}

The author acknowledges the use of AI tools in preparing this paper, as  the arguments here were developed in a dialogue between the author and ChatGPT-5.6 Pro. 

The starting point for this paper originated in the author's earlier work \cite{GuerraTione2025} with R.\ Tione, and the intuition that, given the simple form of $L$, it should be possible to give a geometric interpretation of the quasiconvexity inequality for $L$ through Morse Theory, as is done in \cite{Sverak1992}. A further intuition, which the author learned from many discussions with D. Faraco and L.\ Székelyhidi, is that $L$ is closely related to separation properties of sets in $\R^{2\times 2}$, and one should see it as a distinguished member of the larger family described in Section \ref{sec:sep} in order to find the correct proof. It is this viewpoint which clarifies the role of the  matrices $\pm\Id$ in \eqref{eq:keyid}.

The author then fed \cite{GuerraTione2025,Sverak1992,Szekelyhidi2005} to ChatGPT and, through several prompts, it arrived at an outline of the proof described in the introduction. From the author's point of view, the main novel ideas brought forth by the AI were  the simple but key identity \eqref{eq:keyid}---which, however, was not expressed in terms of the $\detpm,\detpp$ integrands, making it quite hard to parse---and in pointing the author to the existing literature on separating saddle points. The concise formulation of the required results from Morse Theory, described in Section \ref{sec:morse}, was obtained by the author essentially without AI assistance, and the details of the rest of the argument were filled by the author with some AI assistance. The author then generalized the proof described in the introduction to the whole family of separating integrands described in Theorem \ref{thm:main}. 

The text in this paper was written and checked entirely by the author, who takes full responsibility for its correctness.

\section{Preliminaries on Morse theory}
\label{sec:morse}

Recall that, for a symmetric matrix,  its \textit{index} is
the number of its negative eigenvalues, counted with multiplicity.
The index of a non-degenerate critical point $x$ of a function
$\phi\in C^2(\R^n)$ is the index of $\D^2\phi(x)$.

We first recall the classical Morse lemma, which the reader can find e.g.\ in \cite[Lemma~2.2]{Milnor1963}.

\begin{lemma}[Morse lemma]
\label{lemma:morse}
Let $\phi\in C^2(U)$, where $U\subset\mathbb R^n$ is open, and let
$x\in U$ be a non-degenerate critical point of index $j$.
There exist a neighbourhood $V$ of $x$ and $C^1$ coordinates
$(\xi_1,\ldots,\xi_n)$ on $V$, centred at $x$, in which
\[
\phi(\xi)=\phi(x)-\sum_{i=1}^{j}\xi_i^2
       +\sum_{i=j+1}^{n}\xi_i^2.
\]
\end{lemma}
Recall that a \textit{Morse function} is a  function which only has non-degenerate critical points.
The Morse lemma leads to the following definition:

\begin{definition}
\label{definition:separating-saddle}
Let $\phi\in C^2(\mathbb R^n)$ and let $x\in\mathbb R^n$ be a non-degenerate critical point of index $1$. Take the sublevel set
\[
\mathcal N_x\equiv\{z\in\mathbb R^n:\phi(z)<\phi(x)\}.
\]
By the Morse lemma, there is a
neighbourhood $U$ of $x$ such that $\mathcal N_x\cap U$ has exactly
two connected components, denoted by $\mathcal N_x^+$ and
$\mathcal N_x^-$, and
$x\in\overline{\mathcal N_x^+}\cap
  \overline{\mathcal N_x^-}.$
We call
$x$ a \emph{separating saddle point} of $\phi$ if
$\mathcal N_x^+$ and $\mathcal N_x^-$ are contained in different
connected components of the global sublevel set $\mathcal N_x$.
\end{definition}

To the best of our knowledge, Definition \ref{definition:separating-saddle} was first introduced in \cite[Definition 4.1]{Herau2011}, see also \cite[\S 2]{Michel2019}.
It is easy to see that the definition is independent of the choice of $U$.

For a function $\phi\in C^2(\R^n)$, let
\[
k_\phi(a)\equiv\#\left\{\text{connected components of } \{\phi\leq a\}\right\}.
\]
For a generic function, separating saddle points give precisely values at which this number drops:

\begin{lemma}
\label{lemma:separating-jump}
Let $\phi\in C^2(\mathbb R^n)$ be a Morse function such that
\begin{enumerate}
\item\label{assumption:coercive} $\phi(x)\to+\infty$ as $|x|\to\infty$;
\item\label{assumption:critical-points} $\phi$ has finitely many critical
points, with distinct critical values.
\end{enumerate}
If $[a,b]$ contains no critical value of $\phi$, then
\[
k_\phi(a)=k_\phi(b).
\]
 Let $c$ be a critical value and let $x$ be the unique critical point with
$\phi(x)=c$. By assumption~\ref{assumption:critical-points}, choose
$\varepsilon>0$ so that $c$ is the only
critical value in $[c-\varepsilon,c+\varepsilon]$. Then
\[
k_\phi(c+\varepsilon)-k_\phi(c-\varepsilon)
=
\begin{cases}
 1,&\text{if $x$ is a local minimum},\\
-1,&\text{if $x$ is a separating saddle point},\\
 0,&\text{otherwise}.
\end{cases}
\]
\end{lemma}

\begin{proof}
Suppose first that $[a,b]$ contains no critical value. By
assumption~\ref{assumption:coercive}, the set
$\{a\leq\phi\leq b\}$ is compact, and $|\D\phi|$ is bounded away
from zero there. The negative gradient flow therefore gives a
deformation retraction of $\{\phi\leq b\}$ onto
$\{\phi\leq a\}$, see \cite[Theorem~3.1]{Milnor1963}. In particular,
$k_\phi(a)=k_\phi(b)$.

Now let $j$ be the index of $x$ and choose a Morse coordinate
neighbourhood $U$ of $x$. Again by
assumption~\ref{assumption:coercive}, the set
\[
\{c-\varepsilon\leq\phi\leq c+\varepsilon\}\setminus U
\]
is compact, and $\D\phi$ does not vanish there. Thus the negative
gradient flow of $\phi$ carries each point of this set to the level
$c-\varepsilon$, unless it first enters $U$. Consequently, any change
in the connected components of the sublevel sets between
$c-\varepsilon$ and $c+\varepsilon$ is determined by the local model
\[
\phi(\xi)=c-|\xi_-|^2+|\xi_+|^2,
\qquad \xi_-\in\mathbb R^j,
\quad \xi_+\in\mathbb R^{n-j}.
\]
If $j=0$, then $\{\phi \leq c+\e\}$ has one more connected component than $\{\phi\leq c-\e\}$, so $k_\phi$ increases
by one. If $j\geq2$, then $\mathcal N_x\cap U$ is connected, since
$\mathbb S^{j-1}$ is connected, and $k_\phi$ is unchanged. If $j=1$,
then $\mathcal N_x\cap U$ has two components, which become connected
in $\{\phi<c+\e\}$. Hence $k_\phi$ decreases by one exactly when
these two local components lie in different components of
$\{\phi<c\}$, and is unchanged otherwise. Thus, by Definition~\ref{definition:separating-saddle}, $k_\phi$ decreases by one exactly when
$x$ is a separating saddle point. 
\end{proof}

As a simple consequence, cf.\ \cite[Proposition 5.2]{Herau2011}, we find a relationship between the multiplicity of local minima and separating saddle points for a generic Morse function:

\begin{corollary}
\label{corollary:morse-count}
Let $\phi\in C^2(\mathbb R^n)$ be as in 
Lemma~\ref{lemma:separating-jump}. 
If $n_0$ is the number of
local minima and $s$ the number of separating saddle points of $\phi$, then
\[
s=n_0-1.
\]
\end{corollary}

\begin{proof}
By Lemma~\ref{lemma:separating-jump}, $k_\phi$ increases by one when a
local minimum is crossed, decreases by one when a separating saddle
point is crossed, and is unchanged at every other critical value and between critical values.
Then \ref{assumption:coercive}  gives $k_\phi=0$ below
the lowest critical value while, by \ref{assumption:coercive}-\ref{assumption:critical-points}, for sufficiently large $a$ the set $\{\phi\leq a\}$ is connected. Indeed, it suffices to take $a$ to be the maximal critical value: if $\{\phi\leq a\}$ had two connected components then by \ref{assumption:coercive} we could apply the Mountain Pass Lemma between the minimum of $\phi$ in each component to get a critical value above $a$, see e.g.\ \cite{Ambrosetti2007}.
Hence we have $k_\phi(a)=1$ for large $a$ and the total change of $k_\phi$ is
$1=n_0-s,$
as desired.
\end{proof}

\section{Separating functions}
\label{sec:sep}

Let $\gamma\subset\mathbb C$ be a Jordan curve, let $\Omega$ be the bounded component of $\C\setminus \gamma$, and let
$H\colon\C\to\mathbb C$ be $1$-Lipschitz. Consider the curve
\[
\Gamma\equiv\{(z,H(z)):z\in\gamma\}\subset\mathbb R^{2\times2}
\]
and, following \cite{Szekelyhidi2005}, define $L_\Gamma\colon\mathbb R^{2\times2}\to\mathbb R$ by
\begin{equation}
\label{eq:L-Gamma}
L_\Gamma(A)\equiv\det A+
\begin{cases}
\displaystyle
\max\left\{0,\sup_{B\in\Gamma}\left(-\det(A-B)\right)\right\}
&\text{if }a_+\in\overline\Omega,\\
\displaystyle
\sup_{B\in\Gamma}\left(-\det(A-B)\right)
&\text{if }a_+\notin\Omega.
\end{cases}
\end{equation}

We give here two specially important examples of such functions to help the reader digest the definition. Strictly speaking, the first example does not fit in the above setting, since the curve is a straight line, but it is very simple nonetheless.

\begin{example}
For
$$\Gamma=\textup{span}  \begin{bmatrix} 0 & 1 \\ -1 & 0 \end{bmatrix},$$
that is, $\gamma=i\R$ and $H=0$, let us take for definiteness $\Omega=\{z\in\C:\Re z>0\}$. Writing $\det A=|a_+|^2-|a_-|^2$ and expanding the square, one computes
\[
\sup_{B\in\Gamma}\bigl(-\det(A-B)\bigr)
=|a_-|^2-(\Re a_+)^2.
\]
Consequently, if $A$ is symmetric, i.e.\ if $a_+\in\R$, one finds that $L_\Gamma(A)= \detpp(A)$.
\end{example}

\begin{example}
For $\Gamma=\mathrm{SO}(2)$, that is, $\gamma=\mb S^1$ and $H=0$,  one computes that
\[
\sup_{B\in\mathrm{SO}(2)}\bigl(-\det(A-B)\bigr)
=|a_-|^2-\bigl||a_+|-1\bigr|^2.
\]
Consequently $L_{\mathrm{SO}(2)}$ coincides with the function $L$ defined in \eqref{eq:L}. 
\end{example}
Instead of Theorem~\ref{thm:L}, we will prove the following more general result:

\begin{theorem}
\label{thm:main}
Let $\Gamma$ be a $1$-Lipschitz graph over the Jordan curve
$\gamma\subset\C$, as defined above. For every
$A=(a_+,a_-)\in\mathbb R^{2\times2}$ and every
$u\in C_c^\infty(\mathbb R^2)$,
\[
\int_{\mathbb R^2}
\bigl[L_\Gamma(A+\D^2u)-L_\Gamma(A)\bigr]\d x\geq0.
\]
\end{theorem}

We begin with a simple but key algebraic inequality. In the model case where $\Gamma=\tp{SO}(2)$ this becomes an equality, see already  Remark~\ref{rem:so2-decomposition}.

\begin{lemma}
\label{lemma:decomposition}
Let $(\alpha,\beta)$ be a component of the intersection of $\Omega$
with a line parallel to the real axis, ordered so that
$\beta-\alpha>0$, and take the points $P,Q\in\Gamma$ given by
\[
P\equiv(\alpha,H(\alpha)),
\qquad
Q\equiv(\beta,H(\beta)).
\]
For every $A=(a_+,a_-)$ with $\Im a_+=\Im\alpha$, the matrices
$A-P$ and $Q-A$ are symmetric and
\begin{equation}
\label{eq:decomposition}
\begin{aligned}
L_\Gamma(A)-\det A\geq{}
&\detpp(A-P)-\det(A-P)\\
& +\detpp(Q-A)-\det(Q-A)\\
&-\min\bigl\{\detpm(A-P),\detpm(Q-A)\bigr\}.
\end{aligned}
\end{equation}
\end{lemma}

\begin{proof}
Let $M\equiv A-P,N\equiv Q-A$ and note that both matrices are symmetric, since their conformal parts are real. Also, since $P,Q\in \Gamma$ and $H$ is 1-Lipschitz, we see that 
\begin{equation}
\label{eq:sum}
Q-P=M+N\geq 0.
\end{equation}
Since $P,Q\in\Gamma$, from the definition of $L_\Gamma$ we have
\begin{equation}
\label{eq:trivialowerbound}
L_\Gamma(A)-\det A
\geq
\max\{-\det M, -\det N\}.
\end{equation}

If, say, $M<0$, then \eqref{eq:sum} gives
$N=(Q-P)-M\geq-M>0$. Hence
$\det N\geq\det(-M)=\det M$, and the right-hand side of
\eqref{eq:decomposition} is $-\det M$, which equals the maximum in
\eqref{eq:trivialowerbound}. The case $N<0$ is identical and, by
\eqref{eq:sum}, the two cases cannot occur simultaneously.

Suppose next that neither matrix is negative definite. Then
\[
\detpp(M)-\det M=\detpm(M),
\qquad
\detpp(N)-\det N=\detpm(N),
\]
and the right-hand side of \eqref{eq:decomposition} becomes
\[
\detpm(M)+\detpm(N)-\min\{\detpm(M),\detpm(N)\}
=\max\{\detpm(M),\detpm (N)\}.
\]
Thus, if at least one matrix is not positive semidefinite, then this equals $\max\{-\det M,-\det N\}$,
so the conclusion again follows from \eqref{eq:trivialowerbound}.

It remains to consider the case $M,N\geq0$, and then the right-hand side of
\eqref{eq:decomposition} vanishes. In this case $a_+\in [\alpha,\beta]\subset\overline\Omega$ and the conclusion follows from the first branch
in \eqref{eq:L-Gamma}. 
\end{proof}

\begin{remark}\label{rem:so2-decomposition}
If $\Gamma=\mathrm{SO}(2)$ and $a_+\in\mathbb R$, a point of
$\mathbb S^1$ nearest to $a_+$ is one of $\{-1,1\}$. Hence
\[
\sup_{B\in\mathrm{SO}(2)}\bigl(-\det(A-B)\bigr)
=\max\{-\det(A+\Id),-\det(\Id-A)\}.
\]
One can then check that, in this case, \eqref{eq:decomposition} is an equality when
$\alpha=-1$ and $\beta=1$. In fact, as
$\det(A+\Id)+\det(\Id-A)=2+2\det A,$
one even has the simpler identity in \eqref{eq:keyid}.
\end{remark}

\section{Consequences of the area formula}

The purpose of this section is to prove the following result, which will relate the first two terms on the right-hand side of \eqref{eq:decomposition} with the last term:
\begin{proposition}
\label{proposition:area}
Let $\phi\in C^\infty(\mathbb R^2)$ be such that
$\phi(x)=\frac12\langle Ax,x\rangle$ outside a compact set,  for some
$0<A\in \R^{2\times 2}_\tp{sym}$. Then
\begin{equation}
\label{eq:area-separating}
\int_{\mathbb R^2}
\bigl[\detpp(\D^2\phi(x))-\det A\bigr]\d x
=\int_M\detpm(\D^2\phi(x))\d x,
\end{equation}
where we define
\begin{equation}
\label{eq:defphiy}
\phi^y(x)\equiv\phi(x)-y\cdot x,
\end{equation}
where we write $T\equiv \D \phi$, and $M$ is the measurable set
$$M\equiv\{x\in \R^2:x\text{ is a separating saddle point of }
      \phi^{T(x)}\}.$$
\end{proposition}

As we will see, Proposition \ref{proposition:area} follows from the area formula together with  Corollary~\ref{corollary:morse-count}.
Let us begin with a simple lemma:

\begin{lemma}\label{lemma:genericity}
In the setting of Proposition \ref{proposition:area}, for a.e.\ $y\in\mathbb R^2$ the function $\phi^y$
has finitely many non-degenerate critical points and distinct critical values.
\end{lemma}

\begin{proof}
Write
$Z\equiv\{x:\det\D^2\phi(x)=0\}$, so 
$|T(Z)|=0$ by the area formula. If $y\notin T(Z)$, then every
critical point of $\phi^y$ is non-degenerate, hence by the Inverse Function Theorem $T^{-1}(y)$
is discrete. Since $|T(x)|\to \infty$ as $|x|\to \infty$,  the set $T^{-1}(y)$ is also compact, and thus it is finite.

It remains to exclude repeated critical values. Fix
$y_0\notin T(Z)$ and write
$T^{-1}(y_0)=\{p_1,\ldots,p_N\}.$
By the Inverse Function Theorem, there is a
neighbourhood $V$ of $y_0$ and smooth maps
$x_i\colon V\to\mathbb R^2$ such that
\[
T^{-1}(y)=\{x_1(y),\ldots,x_N(y)\},
\qquad \forall y\in V.
\]
The corresponding critical values
$c_i(y)\equiv\phi^y(x_i(y))
=\phi(x_i(y))-y\cdot x_i(y)$
satisfy
\[
\D c_i(y)=-x_i(y),
\qquad
\D(c_i-c_j)(y)=x_j(y)-x_i(y)\neq0
\quad\text{if }i\neq j.
\]
Thus each set $\{c_i=c_j\}$ is either empty or a smooth hypersurface
in $V$, and so has measure zero. A countable cover of
$\mathbb R^2\setminus T(Z)$ by such neighbourhoods shows that the set
of $y$ for which two critical values of $\phi^y$ coincide has zero measure.
\end{proof}

Next, we note that the terms on the first two lines of the right-hand side in \eqref{eq:decomposition} have a simple geometric interpretation:

\begin{lemma}
\label{lemma:degree-count}
Let $\phi,\phi^y$ be as in Lemma~\ref{lemma:genericity} and $n_0(y)$ be the number of local minima of $\phi^y$.
Then
\begin{equation}
\label{eq:degree-count}
\int_{\mathbb R^2}
\bigl[\detpp(\D^2\phi(x))-\det A\bigr]\d x
=\int_{\mathbb R^2}\bigl(n_0(y)-1\bigr)\d y.
\end{equation}
\end{lemma}

\begin{proof}
For $j=0,1,2$, set
\[
E_j\equiv\{x:\det\D^2\phi(x)\neq0,
                 \ \D^2\phi(x)\text{ has index }j\}
\]
and let
$n_j(y)\equiv\#(T^{-1}(y)\cap E_j)$. On $E_1$ the quantity
$\detpp(\D^2\phi)-\det\D^2\phi$ equals
$\lvert\det\D^2\phi\rvert$, while on $E_2$ it equals
$-\lvert\det\D^2\phi\rvert$, and it vanishes elsewhere.
Since the determinant is a null Lagrangian and
$\phi-\frac12\langle A\,\cdot,\cdot\rangle$ is compactly supported,
the area formula therefore gives
\begin{equation}
\label{eq:degree-area-count}
\begin{aligned}
\int_{\mathbb R^2}
 \bigl[\detpp(\D^2\phi)-\det A\bigr]\d x & 
 = \int_{\mathbb R^2}
 \bigl[\detpp(\D^2\phi)-\det\D^2\phi\bigr]\d x\\
& =\int_{E_1}\lvert\det\D^2\phi\rvert\, \d x
-\int_{E_2}\lvert\det\D^2\phi\rvert\, \d x\\
& =\int_{\mathbb R^2}\left(n_1(y)-n_2(y)\right)\d y.
\end{aligned}
\end{equation}

Fix  $y\not \in T(Z)$, and choose $R>0$ so large that
$T^{-1}(y)\subset B_R$, $T(x)=Ax$ for $x\in \partial B_R$, and
$y\in A(B_R)$. Since the topological degree depends only on the boundary values, see \cite{Fonseca1995}, and
$\det A>0$,
we have
$\deg(T,B_R,y)=1,$
and therefore
\begin{align*}
1 =\deg(T,B_R,y) & =\sum_{x\in T^{-1}(y) \cap B_R} \operatorname{sign}\left(\det\D^2\phi(x)\right)\\
& =  \sum_{x\in T^{-1}(y)}\operatorname{sign}\left(\det\D^2\phi(x)\right)
  =n_0(y)-n_1(y)+n_2(y)
\end{align*}
for almost every $y$. Substituting this identity into
\eqref{eq:degree-area-count} proves
\eqref{eq:degree-count}.
\end{proof}

The next lemma, in turn, relates the number of separating saddle points with $\detpm$:

\begin{lemma}
\label{lemma:area}
Let $\phi,\phi^y$ be as in Lemma~\ref{lemma:genericity} and let
 $s(y)$ denote the number of separating saddle points of $\phi^y$.
Then $M$ is a measurable set with $M\subset\{\det\D^2\phi<0\}$, and
\begin{equation}
\label{eq:detpm=s}
\int_M\detpm(\D^2\phi(x))\,\d x
=\int_{\mathbb R^2}s(y)\,\d y.
\end{equation}
\end{lemma}

\begin{proof}
To see that $M$ is measurable, set
$\Sigma\equiv\{x:\det \D^2\phi(x)<0\}.$
The set $\Sigma$ is open, and the set $N\equiv\Sigma\setminus M$ of
non-separating saddle points is open in $\Sigma$. Indeed, fix
$x_0\in N$ and let
\[
\psi_x(z)\equiv
\phi(z)-\phi(x)-\D\phi(x)\cdot(z-x) = \phi^{T(x)}(z)-\phi^{T(x)}(x),
\]
so that $\psi_x(x)=0$, $\D\psi_x(x)=0$, and $\psi_x$ differs from $\phi^{T(x)}$ by an additive constant.  Since $x_0$ is a
non-separating saddle point, let $\xi,\xi^\bot \in \mb S^1$ be respectively the negative and positive eigenvectors of $\D^2 \phi(x_0)$, so that the points $p_\pm = x_0 \pm \frac \e 2 \xi$ are in the two components
of $\{\psi_{x_0}<0\}\cap B_\e(x_0)$, for $\e>0$ small, and let also 
 $\gamma\subset\{\psi_{x_0}<0\}$ be a path joining them. Since $\gamma$ is
compact, $\max_\gamma\psi_{x_0}<0$. As
$\psi_x\to\psi_{x_0}$ in $C^\infty_\loc$ when $x\to x_0$, the same
path lies in $\{\psi_x<0\}$ for all $x$ sufficiently close to $x_0$,
while at the same time $p_+,p_-$ remain in the two distinct components of
$\{\psi_x<0\}\cap B_\e(x)$: indeed, by Taylor's Theorem, and shrinking $\e$ if needed, the line $[x -\e \xi^\bot,x+\e \xi^\bot]\subset \{\psi_x\geq 0\}$ separates $p_+$ and $p_-$. Thus $x$ is still a non-separating saddle point, so
$N$ is open in $\Sigma$ and hence $M=\Sigma\setminus N$ is
measurable.

For every $y\in\mathbb R^2$, the definition of $M$ gives
\[
\#\bigl(T^{-1}(y)\cap M\bigr)=s(y).
\]
Since clearly $M\subset\{\det\D^2\phi<0\}$, the area formula then yields
\[
\int_M\detpm(\D^2\phi(x))\,\d x
=\int_M\lvert\det \D T(x)\rvert\,\d x
=\int_{\mathbb R^2}s(y)\,\d y,
\]
as wished.
\end{proof}

\begin{proof}[Proof of Proposition \ref{proposition:area}]
By Lemma~\ref{lemma:genericity}, for almost every $y$ the function
$\phi^y$ is Morse and satisfies condition \ref{assumption:critical-points} in Lemma \ref{lemma:separating-jump}. It also clearly satisfies condition \ref{assumption:coercive}, so
Corollary~\ref{corollary:morse-count} applies and gives
$s(y)=n_0(y)-1$. The conclusion follows by
combining \eqref{eq:degree-count} and \eqref{eq:detpm=s}.
\end{proof}

\section{Proof of Theorem~\ref{thm:main}}

To prove Theorem \ref{thm:main}, we just require the following planar topology argument:

\begin{lemma}
\label{lemma:planar-separation}
Let $\psi_1,\psi_2\in C^2(\mathbb R^2)$ and suppose that $x$ is a non-degenerate
saddle point of both functions, with
$\psi_1(x)=\psi_2(x)=0.$
If
\begin{equation}
\label{eq:disjoint}
\{\psi_1<0\}\cap \{\psi_2<0\}=\emptyset,
\end{equation}
then $x$ is a separating saddle point of at least one of $\psi_1$ and $\psi_2$.
\end{lemma}

\begin{proof}
Suppose that $x$ is not a separating saddle point of $\psi_1$. 
After translating, we may assume that $x=0$. By the Morse lemma, there
are local coordinates $\xi=(\xi_1,\xi_2)$ around $0$ such that
\[
    \psi_1(\xi)=-\xi_1^2+\xi_2^2.
\]
Thus the two components of $\{\psi_1<0\}\cap B_\e(0)$, for $\e>0$ small, are
$\{\xi_1>|\xi_2|\}$ and $\{\xi_1<-|\xi_2|\}$. By assumption these sets belong to the
same connected component of  $\{\psi_1<0\}$, so one can find a Jordan curve $\gamma\subset \{\psi_1<0\}\cup \{0\}$
with $0\in \gamma$ and which coincides with the $\xi_1$-axis in a
neighbourhood of $0$.
By \eqref{eq:disjoint},
\[
\{\psi_2<0\}\cap B_\e(0)
\subseteq\{\psi_1\geq0\}\cap B_\e(0)
=\{|\xi_2|\geq|\xi_1|\}\cap B_\e(0).
\]
We can write $\psi_2(\xi) = \langle H\xi, \xi \rangle + o(|\xi|^2)$ as $\xi\to 0$ for some indefinite, invertible symmetric matrix $H$, so we can find $v\in \mb S^1$ such that $\langle Hv,v\rangle<0$ and hence, for $t$ small enough, $\pm tv \in \{\psi_2<0\}\cap B_\e(0)$. 
It follows that the two connected components of
$\{\psi_2<0\}\cap B_\e(0)$ lie in different connected components of
$\{|\xi_2|\geq|\xi_1|\}\setminus\{0\},$
and hence lie on opposite sides of $\gamma$.
Moreover, \eqref{eq:disjoint} implies that
$\gamma\cap\{\psi_2<0\}=\emptyset$. By the Jordan curve theorem, the two components of
$\{\psi_2<0\}\cap B_\e(0)$ belong to different  components of $\{\psi_2<0\}$, thus $0$ is a separating saddle point of $\psi_2$.
\end{proof}

\begin{proof}[Proof of Theorem~\ref{thm:main}]
Fix $u\in C_c^\infty(\mathbb R^2)$.
For $B=(z,H(z))\in\Gamma$,
\[
-\det(A-B)=|a_--H(z)|^2-|a_+-z|^2.
\]
Since $\Gamma$ is compact, the supremum in \eqref{eq:L-Gamma} is
attained. We distinguish two cases.

\noindent\textbf{Case 1.} Suppose that either $a_+\notin\Omega$, or
$a_+\in\Omega$ and
$\sup_{B\in\Gamma}\bigl(-\det(A-B)\bigr)\geq0$, thus
$$L_\Gamma(A) = \det A +\sup_{B\in \Gamma} \bigl(-\det(A-B)\bigr).$$
We choose $B_0\in\Gamma$ attaining the supremum and write
$\ell_B(X)\equiv\det X-\det(X-B)$.   Then, from the definition of $L_\Gamma$, we have the lower bound
\begin{equation}
\label{eq:affine-minorant}
L_\Gamma(X)\geq\ell_B(X)
\qquad\forall X\in\mathbb R^{2\times2},\forall B\in\Gamma,
\end{equation}
with equality at $B=B_0$ and $X=A$.
Using \eqref{eq:affine-minorant} and the fact that $\ell_{B_0}$ is
affine while $u$ is compactly supported,
\[
\begin{aligned}
\int_{\mathbb R^2}
\bigl[L_\Gamma(A+\D^2u)-L_\Gamma(A)\bigr]\d x
\geq\int_{\mathbb R^2}
\bigl[\ell_{B_0}(A+\D^2u)-\ell_{B_0}(A)\bigr]\d x=0.
\end{aligned}
\]

\medskip
\noindent\textbf{Case 2.} Suppose that the conditions of Case 1 do not hold, i.e.\ that $a_+\in \Omega$ and 
\begin{equation}
\label{eq:strictineq}
|a_--H(z)|<|a_+-z|
\quad\forall z\in\gamma.
\end{equation}
Let $(\alpha,\beta)$ 
be the component of
$\Omega\cap(a_++\mathbb R)$ containing $a_+$, ordered so that
$\beta-\alpha>0$, and let 
$P, Q$ be as in Lemma \ref{lemma:decomposition}, i.e.\
\[
P=(\alpha,H(\alpha)),
\qquad
Q=(\beta,H(\beta)).
\]
By \eqref{eq:strictineq}, $A-P>0$ and $Q-A>0$. Define
\[
\phi_P(x)\equiv\frac12\langle(A-P)x,x\rangle+u(x),
\qquad
\phi_Q(x)\equiv\frac12\langle(Q-A)x,x\rangle-u(x),
\]
and let $M_P,M_Q$ be the corresponding sets given by
Lemma~\ref{lemma:area}.

By the assumptions of this case, $L_\Gamma(A)=\det A$.
We apply Lemma~\ref{lemma:decomposition} with $A+\D^2u(x)$ in place of
$A$. Since
$A+\D^2u-P=\D^2\phi_P,$
$Q-A-\D^2u=\D^2\phi_Q,$
and the determinant is a null Lagrangian,  we have
\begin{equation}
\label{eq:key-reduction}
\begin{aligned}
&\int_{\mathbb R^2}
  \bigl[L_\Gamma(A+\D^2u)-L_\Gamma(A)\bigr]\d x\\
&\quad\geq
\int_{\mathbb R^2}
\bigl[\detpp(\D^2\phi_P)-\det(\D^2\phi_P)\bigr]\d x
+
\int_{\mathbb R^2}
\bigl[\detpp(\D^2\phi_Q)-\det(\D^2\phi_Q)\bigr]\d x\\
&\qquad-
\int_{\mathbb R^2}
\min\bigl\{\detpm(\D^2\phi_P),
             \detpm(\D^2\phi_Q)\bigr\}\d x.
\end{aligned}
\end{equation}
Using again the null Lagrangian property, combined with
Proposition~\ref{proposition:area}, gives
\begin{equation*}
\label{eq:two-area-identities}
\begin{aligned}
\int_{\mathbb R^2}
\left[\detpp(\D^2\phi_P)-\det(\D^2\phi_P)\right]\d x
&=\int_{M_P}\detpm(\D^2\phi_P)\d x,\\
\int_{\mathbb R^2}
\left[\detpp(\D^2\phi_Q)-\det(\D^2\phi_Q)\right]\d x
&=\int_{M_Q}\detpm(\D^2\phi_Q)\d x.
\end{aligned}
\end{equation*}
Consequently, \eqref{eq:key-reduction} becomes
\begin{equation}
\label{eq:reduction}
\begin{aligned}
\int_{\mathbb R^2}
  \bigl[L_\Gamma(A+\D^2u)-L_\Gamma(A)\bigr]\d x
& \geq 
\int_{\mathbb R^2}\Bigl[
1_{M_P}\detpm(\D^2\phi_P)
+1_{M_Q}\detpm(\D^2\phi_Q)
\\
& \qquad -\min\bigl\{\detpm(\D^2\phi_P),
              \detpm(\D^2\phi_Q)\bigr\}\Bigr]\d x.
\end{aligned}
\end{equation}

Thus the claim will follow from the pointwise inequality
\begin{equation}
\label{eq:pointwise-general}
\min\bigl\{\detpm(\D^2\phi_P),\detpm(\D^2\phi_Q)\bigr\}
\leq
1_{M_P}\detpm(\D^2\phi_P)+1_{M_Q}\detpm(\D^2\phi_Q)
\end{equation}
which holds pointwise in $\R^2$. It suffices to prove
\eqref{eq:pointwise-general} in the set
\[
\mathcal O\equiv
\bigl\{x:\detpm(\D^2\phi_P(x))>0,
          \ \detpm(\D^2\phi_Q(x))>0\bigr\}.
\]
Fix $x\in\mathcal O$ and define
\[
\begin{aligned}
\psi_P(z)&\equiv\phi_P(z)-\phi_P(x)-\D\phi_P(x)\cdot(z-x),\\
\psi_Q(z)&\equiv\phi_Q(z)-\phi_Q(x)-\D\phi_Q(x)\cdot(z-x),
\end{aligned}
\]
which differ from the functions $\phi_P^{\D \phi_P(x)}$ and $\phi_Q^{\D \phi_Q(x)}$, defined through \eqref{eq:defphiy}, by a constant.
Now $x$ is a non-degenerate saddle point of both $\psi_P, \psi_Q$. Note that 
$\phi_P(z)+\phi_Q(z)
=
\frac12\langle(Q-P)z,z\rangle$.
Thus the inequality $Q-P=(Q-A)+(A-P)\geq0$ gives
\begin{align*}
\psi_P(z)+\psi_Q(z)& = 
(\phi_P+\phi_Q)(z)-(\phi_P+\phi_Q)(x)-\langle \D(\phi_P+\phi_Q)(x),z-x\rangle\\
& = 
\frac12\langle (Q-P)z,z\rangle
-\frac12\langle (Q-P)x,x\rangle
-\langle (Q-P)x,z-x\rangle
\\
& =\frac12\langle(Q-P)(z-x),z-x\rangle\geq0.
\end{align*}
It follows that \eqref{eq:disjoint} holds for $\psi_P,\psi_Q$, so by Lemma~\ref{lemma:planar-separation} we conclude that $x$ is a separating
saddle point of at least one of $\psi_P$ and $\psi_Q$. Thus
$x\in M_P\cup M_Q$ and \eqref{eq:pointwise-general} holds on $\mc O$.
\end{proof}

\section{Proofs of the corollaries}

\begin{proof}[Proof of Corollary~\ref{cor:beurling}]
For $A=(a_+,a_-)$ and $\lambda>0$, the definition of $L$ gives
\begin{equation}
\label{eq:weak-obstacle}
L(A/\lambda)
\leq \frac{2}{\lambda}|a_+|-1_{\{|a_-|>\lambda\}}.
\end{equation}
Indeed, if $|A|\leq\lambda$, the indicator vanishes and
$|a_+|^2-|a_-|^2= |A|(|a_+|-|a_-|)\leq \lambda|a_+|$ so \eqref{eq:weak-obstacle} follows, while if $|A|>\lambda$ then \eqref{eq:weak-obstacle} follows immediately from the second branch in \eqref{eq:L}.

By density it suffices to take $h\in C_c^\infty(\mathbb C;\mathbb R)$, and we then define
\[
u(x)\equiv\frac1\pi\int_{\mathbb C}\log|x-y|h(y)\d y.
\]
Then $\Delta u=2h$ and
\[
u(x)=O(\log|x|),
\qquad
\D u(x)=O(|x|^{-1}),
\qquad
\D^2u(x)=O(|x|^{-2}).
\]
Let $\eta_R=1$ on $B_R$ and $\eta_R=0$ outside $B_{2R}$, with
$|\D^j\eta_R|\leq CR^{-j}$ for $j=1,2$, thus
\begin{equation}
\label{eq:cutoffdecay}
\D^2(\eta_Ru)=O\left(\frac{1+\log R}{R^2}\right).
\end{equation}
For fixed $\lambda>0$, both $\D^2 u$ and $\D^2(\eta_R u)$ lie in the quadratic branch of
$L$ respectively on $B_{2R}\setminus B_R$ and $\R^2\setminus B_R$, for $R$ large enough. Consequently we have
\[
\int_{B_{2R}\setminus B_R}
\left|L\bigl(\D^2(\eta_Ru)/\lambda\bigr)\right|\d x
=O\left(\frac{(1+\log R)^2}{R^2}\right),
\qquad
\int_{\mathbb R^2\setminus B_R}
\left|L(\D^2u/\lambda)\right|\d x=O(R^{-2}),
\]
and, using Theorem \ref{thm:L}, we infer that
\begin{equation}
\label{eq:cutoffconcl}
0 \leq \int_{\mathbb R^2}L\left(\D^2(\eta_Ru)/\lambda\right)\d x
\to\int_{\mathbb R^2}L(\D^2u/\lambda)\,\d x.
\end{equation}
Note that $(\D^2 u)_-=\frac12 (u_{xx}-u_{yy}) + i u_{xy},$ hence $2u_{zz}=\overline{(\D^2 u)_-}$ and
$$\mc S h = \mc S(\tfrac 1 2 \Delta u) =2 \mc S( u_{z \bar z}) = 2 u_{zz} = \overline{(\D^2 u)_-}
\quad \implies \quad
|\mc S h| = |(\D^2 u)_-|.
$$
Since also $(\D^2 u)_+ = 2 u_{z \bar z} = h$,
combining
\eqref{eq:weak-obstacle}--\eqref{eq:cutoffconcl}  gives
$$
0 \leq \int_{\R^2} L(\D^2 u/\lambda) \, \d x\leq \frac 2 \lambda \|h\|_{L^1}-\{|\mc S h|>\lambda\}|,
$$
and the desired estimate follows. 

To prove sharpness, set $c_n\equiv1+2/n$, and consider
the probability measure
\begin{align*}
\nu_n&\equiv
\frac{1-1/n}{2}\bigl(\delta_{\operatorname{diag}(c_n,-1)}
 +\delta_{-\operatorname{diag}(c_n,-1)}\bigr)
+\frac1{2n}\bigl(\delta_{\operatorname{diag}(c_n,n-1)}
 +\delta_{-\operatorname{diag}(c_n,n-1)}\bigr)\\
&=\frac{1-1/n}{2}
 \left(\delta_{(1/n,1+1/n)}+\delta_{-(1/n,1+1/n)}\right)
 +\frac1{2n}
 \left(\delta_{(n/2+1/n,1+1/n-n/2)}
 +\delta_{-(n/2+1/n,1+1/n-n/2)}\right),
\end{align*}
where the second line is in complex coordinates.
This is a second-order laminate with barycentre zero: we first split zero
equally into $\pm\operatorname{diag}(c_n,0)$ and then we split each matrix in
the second coordinate with weights $1-1/n$ and $1/n$, see Figure \ref{fig:laminate} as well as  \cite{Muller1999a} for further details on laminates.  Moreover, for $n\geq 5$, we have
\[
\supp\nu_n\subset\{|a_-|\geq1\},
\qquad
\int|a_+|\,\d\nu_n=\frac12+\frac1n = \frac{n + 2}{2n}	.
\]
The realization of finite-order laminates by Hessians of compactly
supported smooth potentials in \cite[Corollary~1]{Boros2013} gives
real-valued conformal parts $h_{n,k}$ such that
\[
\frac{\bigl|\{|\mathcal S h_{n,k}|>1\}\bigr|}
     {\|h_{n,k}\|_{L^1}}
\to\frac{2n}{n+2}
\qquad\text{as }k\to\infty.
\]
Letting $n\to\infty$ proves that the constant $2$ is sharp.
\end{proof}

\begin{figure}
\begin{tikzpicture}[scale=1.0,>=stealth]

\def\n{6} 

\pgfmathsetmacro{\cn}{1 + 2/\n}

\pgfmathsetmacro{\Px}{\cn/2}
\pgfmathsetmacro{\Py}{\cn/2}

\pgfmathsetmacro{\Aax}{1/\n}
\pgfmathsetmacro{\Aay}{1 + 1/\n}

\pgfmathsetmacro{\Abx}{\n/2 + 1/\n}
\pgfmathsetmacro{\Aby}{1 + 1/\n - \n/2}

\pgfmathsetmacro{\Bax}{-\Aax}
\pgfmathsetmacro{\Bay}{-\Aay}
\pgfmathsetmacro{\Bbx}{-\Abx}
\pgfmathsetmacro{\Bby}{-\Aby}

\draw[->] (-4.2,0) -- (4.2,0) node[right] {$a_+$};
\draw[->] (0,-3.2) -- (0,3.2) node[above] {$a_-$};

\draw[densely dashed,gray!50] (-3.8,-3.8) -- (3.8,3.8);
\draw[densely dashed,gray!50] (-3.8,3.8) -- (3.8,-3.8);




\fill[black] (\Aax,\Aay) circle (2.2pt)
  node[above right=1pt] {$(\tfrac1n,\,1+\tfrac1n)$};

\fill[black] (\Abx,\Aby) circle (2.2pt)
  node[below right=1pt] {$ (\tfrac n2+\tfrac1n,\,1+\tfrac1n-\tfrac n2)$};

\fill[black] (\Bax,\Bay) circle (2.2pt)
  node[below left=1pt] {$-(\tfrac1n,\,1+\tfrac1n)$};

\fill[black] (\Bbx,\Bby) circle (2.2pt)
  node[above left=1pt] {$ -(\tfrac n2+\frac1n,\,1+\frac1n-\frac n2)$};

\draw[thick,black,->] (0,0) -- (\Px,\Py)
  node[midway,above left=-1pt] {};
\draw[thick,black,->] (0,0) -- (-\Px,-\Py)
  node[midway,below right=-1pt] {};

\draw[thick,black,->] (\Px,\Py) -- (\Aax,\Aay)
  node[midway,above left=-1pt] {};
\draw[thick,black,->] (\Px,\Py) -- (\Abx,\Aby)
  node[midway,below left=-1pt] {};

\draw[thick,black,->] (-\Px,-\Py) -- (\Bax,\Bay)
  node[midway,below right=-1pt] {};
\draw[thick,black,->] (-\Px,-\Py) -- (\Bbx,\Bby)
  node[midway,above right=-1pt] {};

\end{tikzpicture}[]
\caption{The second order laminate $\nu_n$. The dashed axes is $\{\det=0\}.$}
\label{fig:laminate}
\end{figure}
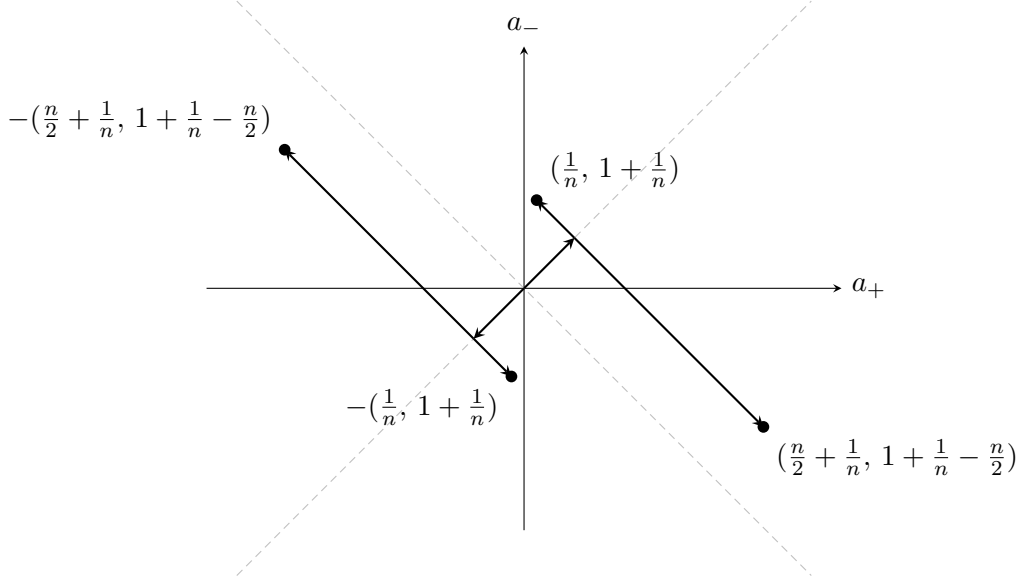

\begin{proof}[Proof of Corollary~\ref{cor:burkholder}]
Set
$
M(X)\equiv L(X)-\det X.$
Clearly Theorem~\ref{thm:L} implies the same Hessian-quasiconvexity inequality for $M$. The key observation, due to \cite[(1.2a)--(1.2b)]{Baernstein1997a}, is that one can recover $B_p$ through a Mellin transform\footnote{The fact that the integral formulas in this proof are Mellin transforms was pointed out to the author by \v Sver\'ak.}:
\begin{align*}
\int_0^\infty t^{p-1}L(A/t)\d t
&=\frac{2}{p(2-p)}B_p(A),&&1<p<2,\\
\int_0^\infty t^{p-1}M(A/t)\d t
&=\frac{2}{p(p-1)(p-2)}B_p(A),&&2<p<\infty.
\end{align*}
For $t>0$, the integrands $A\mapsto L(A/t), M(A/t)$ are
Hessian-quasiconvex, and $B_p$ can be expressed  as a positive weighted combination of these integrands, hence it is Hessian-quasiconvex too.
\end{proof}

\begin{proof}[Proof of Corollary~\ref{cor:beurling-lp}]
Set
$\alpha_p\equiv p\left(1-\frac1{p^*}\right)^{p-1}.$
The pointwise estimate \cite[p.~93]{Baernstein1997a} reads as
\begin{equation}
\label{eq:burkholder-pointwise}
\alpha_p B_p(A)\leq
\begin{cases}
(p^*-1)^p|a_+|^p-|a_-|^p,&1<p<2,\\
(p^*-1)^p|a_-|^p-|a_+|^p,&2<p<\infty.
\end{cases}
\end{equation}
By density we can take $h\in C_c^\infty(\mathbb C;\mathbb R)$, and let $u$ be the
potential defined in the proof of Corollary~\ref{cor:beurling}. For the cutoffs
$\eta_R$ used in that proof, from \eqref{eq:cutoffdecay} it is easy to see that
\[
\D^2(\eta_Ru)\to\D^2u
\quad\text{in }L^p(\mathbb R^2),
\qquad 1<p<\infty.
\]
As in the proof of Corollary \ref{cor:beurling}, we have
$(\D^2u)_+=h$  and $|(\D^2u)_-|=|\mathcal S h|$. Hence, applying Corollary~\ref{cor:burkholder} at $A=0$
to $\eta_Ru$, using \eqref{eq:burkholder-pointwise}, and taking the limit
gives
\begin{align*}
\|\mathcal S h\|_{L^p}
&\leq\frac1{p-1}\|h\|_{L^p},&&1<p<2,\\
\|h\|_{L^p}
&\leq(p-1)\|\mathcal S h\|_{L^p},&&2<p<\infty.
\end{align*}
It is easy see verify that, if $\overline{\cdot}$ denotes the conjugation operator, then
$$\mc S^{-1} \circ \overline{\cdot} = \overline{\cdot} \circ \mc S,$$
as can be checked by using the rules for conjugation of Wirtinger derivatives.
Hence, for a real-valued function $g$, we have $\mc S^{-1}g=\overline{\mc S g}$. 
Thus, in the case $2<p<\infty$, if $g=\mathcal S h$ is real-valued then  we can apply 
the second inequality above to $g$ to obtain
$\|\mathcal S h\|_{L^p}\leq(p-1)\|h\|_{L^p}$, as wished.

It remains to prove sharpness. Fix $1<p<\infty$, $p\neq2$, and
$0<\alpha<1/p$, and consider the radial gradient map
\[
v_\alpha(z)\equiv
\begin{cases}
z|z|^{-2\alpha},&|z|\leq1,\\
1/\overline z,&|z|>1
\end{cases}
\]
which satisfies
\[
\begin{cases}
\partial_zv_\alpha=(1-\alpha)|z|^{-2\alpha}\\
|\partial_{\bar z}v_\alpha|
=\alpha|z|^{-2\alpha}
\end{cases}
\text{in } \mb D, 
\qquad 
\qquad
\begin{cases}
\partial_zv_\alpha=0\\
|\partial_{\bar z}v_\alpha|=|z|^{-2}
\end{cases}
\text{outside } \mb D.
\]
Consequently,
\[
\frac{\|\partial_{\bar z}v_\alpha\|_{L^p}}
     {\|\partial_zv_\alpha\|_{L^p}}
\to\frac1{p-1}
\qquad\text{as }\alpha\nearrow\frac1p.
\]
Now sharpness for $1<p<2$ follows by taking $h=\p_z v_\a$ and for $2<p<\infty$ by taking $h=\p_{\bar z} v_\a$, so that $\mc S h = \p_z v_\a$ is real-valued.
\end{proof}

{\small
\bibliographystyle{abbrv-andre}
\bibliography{../../library}
}

\end{document}